\documentclass[12pt]{amsart}
\usepackage{amssymb, amsmath}
\newtheorem{question}{Question}[section]
\newtheorem{theorem}[question]{Theorem}
\newtheorem{lemma}[question]{Lemma}
\newtheorem{corollary}[question]{Corollary}

\newtheorem{definition}[question]{Definition}

\newtheorem{conjecture}[question]{Conjecture}
\title{On two topological cardinal invariants of an order-theoretic flavour}

\author{Santi Spadaro}
\address{Department of Mathematics, Faculty of Science and Engineering, York University, Toronto, ON,  M3J 1P3 Canada}
\email{sspadaro@mathstat.yorku.ca, santispadaro@yahoo.com}


\thanks{The author was partially supported by the Center for Advanced Studies in Mathematics at Ben Gurion University and by an INdAM-Cofund outgoing fellowship. The author is also grateful to the Fields Institute of the University of Toronto for hospitality.}

\begin{document}

\begin{abstract}
Noetherian type and Noetherian $\pi$-type are two cardinal functions which were introduced by Peregudov in 1997, capturing some properties studied earlier by the Russian School. Their behavior has been shown to be akin to that of the \emph{cellularity}, that is the supremum of the sizes of pairwise disjoint non-empty open sets in a topological space.  Building on that analogy, we study the Noetherian $\pi$-type of $\kappa$-Suslin Lines, and we are able to determine it for every $\kappa$ up to the first singular cardinal.  We then prove a consequence of Chang's Conjecture for $\aleph_\omega$ regarding the Noetherian type of countably supported box products which generalizes a result of Lajos Soukup. We finish with a connection between PCF theory and the Noetherian type of certain Pixley-Roy hyperspaces.
\end{abstract}

\maketitle

\section{Introduction and notation}
In this paper the letter $X$ always denotes a topological space and every topological space is assumed to have at least two points. Cardinals are initial ordinals and they are provided with the order topology. The symbol $cov(\kappa, <\mu)$ denotes the cofinality of the partial order $([\kappa]^{<\mu}, \subseteq)$. All undefined notions can be found in \cite{En} for Topology and \cite{Je} for Set Theory. A \emph{cardinal function} is a function from the class of all topological spaces into the class of all cardinals which is invariant under homeomophisms. Some cardinal functions measure the least size of a given family of subsets of topological spaces. For instance, the \emph{$\pi$-weight of $X$} ($\pi w(X)$) is defined as the minimum size of a $\pi$-base of $X$ (that is, a set $\mathcal{P}$ of non-empty open subsets of $X$, such that for every non-empty open set $U \subset X$, there is $P \in \mathcal{P}$ with $P \subset U$). Other cardinal functions are defined as the supremum of the sizes of certain objects in a topological space. For example, the \emph{cellularity} of $X$ ($c(X)$) is defined as the supremum of the sizes of families of pairwise disjoint non-empty open sets in $X$. The cardinal functions we take up in this paper deviate from this two-faced model and all stem from the following order-theoretic definition. 

\begin{definition}
\cite{MilHom} Let $(P, \leq)$ be a partially ordered set. We say that $P$ is \emph{$\kappa^{op}$-like} if every bounded subset of $P$ has cardinality strictly less than $\kappa$.
\end{definition}

We are interested in the least $\kappa$ such that $X$ has a given family of subsets satisfying the above property under reverse containment. We focus on bases and $\pi$-bases, thus giving rise to the following two \emph{order-theoretic} cardinal functions, which were introduced by Peregudov in \cite{PNt}.

\begin{definition}
Let $X$ be a space. The \emph{Noetherian type of $X$} ($Nt(X)$) is the minimum cardinal $\kappa$ such that $X$ has a $\kappa^{op}$-like base, with respect to $\supseteq$. The \emph{Noetherian $\pi$-type of $X$} ($\pi Nt(X)$) is the minimum cardinal $\kappa$ such that $X$ has a $\kappa^{op}$-like $\pi$-base.
\end{definition}

Spaces of countable Noetherian type were called \emph{Noetherian} in \cite{PS} and \cite{Ma} and \emph{OIF spaces} in \cite{BBBGLM}. Based on their definition, Noetherian type and Noetherian $\pi$-type may seem like order-theoretic versions of the weight and the $\pi$-weight.  However their behavior seem to deviate a lot from their classical counterparts. Spaces of countable Noetherian type can in fact have arbitrarily large $\pi$-weight (simply take a Cantor cube of weight $\kappa$). These cardinal functions rather exhibit a behavior which is closer to the cellularity. Some evidence to this claim is provided by David Milovich's theorem from \cite{MilHom} saying that compact homogeneous dyadic spaces have countable Noetherian type, and by a bound for the Noetherian type in the $G_\delta$ topology which was proved in \cite{KMS}. We wish to elaborate more on this affinity by studying the Noetherian $\pi$-type of higher Suslin Lines in the first section. In \cite{KMS} the authors showed how the Noetherian type of a deceivingly simple space like $(2^{\aleph_\omega}, \tau_{<\omega_1})$, that is, the product of $\aleph_\omega$ many copies of the two-point discrete space, cannot be determined in ZFC. In particular this Noetherian type takes its largest possible value if Chang's Conjecture for $\aleph_\omega$ holds along with the GCH (this was first proved by Soukup in \cite{S}) and its minimum value in the constructible universe. Here we extend this result by replacing the two-point discrete space with a much larger class of spaces, which includes all scattered spaces. We finish with another topological equivalent of a set-theoretic statement introduced in \cite{KMS}, which leads to sharp bounds to the Noetherian type of certain Pixley-Roy Hyperspaces. In particular, we find that the gap between the Noetherian type of a space and the Noetherian type of its Pixley-Roy hyperspace can be very large.
 
\section{On the Noetherian $\pi$-type of higher Suslin Lines}
Recall that a $\kappa$-Suslin Line $L$ is a linearly ordered set satisfying the following properties:

\begin{enumerate}
\item $L$ is \emph{dense}, that is, the order topology on $L$ has no isolated points.
\item $L$ is \emph{complete}, that is, every non-empty bounded subset of $L$ has an infimum and a supremum.
\item $L$ has \emph{cellularity no greater than $\kappa$}, that is no pairwise disjoint family of intervals in $L$ has size $\kappa^+$.
\item $L$ has \emph{density greater than $\kappa$}, that is no set of size $\kappa$ is dense in $L$ with the order topology.
\end{enumerate}

An $\aleph_0$-Suslin Line is thus a usual Suslin Line. The existence of a $\kappa$-Suslin Line is known to be consistent with ZFC, for every cardinal $\kappa$. 

David Milovich (\cite{MilHom}, Theorem 2.26) proved that the Noetherian $\pi$-type of a Suslin Line is uncountable. We modify some arguments of Juh\'asz, Soukup and Szentmikl\'ossy from \cite{JSS} to determine the exact Noetherian $\pi$-type of higher Suslin Lines. We begin with a lemma which must be well-known.

\begin{lemma} \label{continuous}
Let $L$ be a complete dense linear order. Then $L$ contains a dense set of points of countable $\pi$-character.
\end{lemma}

\begin{proof}
Let $(a,b)$ be a non-empty open interval. We claim that $(a,b)$ contains a point of countable $\pi$-character. Let $c \in (a,b)$ and suppose you have constructed $\{x_i: i \leq n \} \subset (a,c)$. Choose $x_{n+1} \in (x_n,c)$. Then $y=\sup \{x_n: n < \omega \} \in (a,b)$ and $\{(x_n,y): n < \omega \}$ is a local $\pi$-base at $y$.
\end{proof}

The argument used by Milovich to prove that every Suslin Line has uncountable Noetherian type extends verbatim to prove the following lemma.

\begin{lemma} \label{lowerbound}
Let $L$ be a $\kappa$-Suslin Line. Then $\pi Nt(L) \geq \aleph_1$.
\end{lemma}

\begin{theorem} \label{upperbound}
Let $L$ be a $\kappa$-Suslin Line. Then $\pi Nt(L) \leq \kappa^+$.
\end{theorem}

\begin{proof}
Let $D \subset L$ be the set of all points of countable $\pi$-character. By Lemma $\ref{continuous}$, $D$ is dense in $L$ and so, $d(D)=\kappa^+$. So there is a dense set $E=\{x_\alpha: \alpha < \kappa^+\} \subset D$ which is left separated in type $\kappa^+$. In other words, for every $\alpha < \kappa^+$ there is a neighbourhood $U_\alpha$ of $x_\alpha$ such that $U_\alpha \cap \{x_\beta: \beta < \alpha \} =\emptyset$. Let $\mathcal{B}_\alpha$ be a countable local base at $x_\alpha$ such that $B \subset U_\alpha$  for every $B \in \mathcal{B}_\alpha$. Then $\mathcal{B}=\bigcup \{\mathcal{B}_\alpha: \alpha < \kappa^+\}$ is a $\pi$-base such that $|\{B \in \mathcal{B}: x_\alpha \in B \}| \leq \kappa$. Now, from this observation and the fact that $E$ is dense in $X$ it follows that $\mathcal{B}$ is $(\kappa^+)^{op}$-like.
\end{proof}

Given a set map $F: S \to \mathcal{P}(S)$, recall that a set is called \emph{free for $F$} if for every $x,y \in S$ with $x \neq y$ we have that $x \notin F(y)$ and $y \notin F(x)$.

\begin{lemma} \label{hajnal}
(Hajnal's Free Set Lemma) Let $S$ be a set and $\kappa$ be a cardinal such that $\kappa < |S|$. Let $F: S \to \mathcal{P}(S)$ be a map such that $|F(x)| < \kappa$ for every $x \in S$. Then there is a set $D \subset S$ such that $|D|=|S|$ and $D$ is free for $F$.
\end{lemma}

\begin{theorem} \label{largediscrete}
Let $X$ be a space such that $\pi Nt(X) <d(X)$, $cf(\pi Nt(X)) \geq \aleph_1$ and $X$ contains a dense set of points of countable $\pi$-character. Then there is a discrete set $D \subset X$ such that $|D| \geq d(x)$.
\end{theorem}

\begin{proof}
Let $\mathcal{B}$ be a $\pi$-base witnessing $\pi Nt(X) < d(X)$ and for every $B \in \mathcal{B}$ choose $x_B \in B$ such that $\pi \chi (x_B,X) \leq \aleph_0$. Let $S=\{x_B: B \in \mathcal{B} \}$. Note that $S$ is dense in $X$ so $|S| \geq d(X)$. For every $x \in S$, fix a countable local $\pi$-base $\{U^x_n: n < \omega \}$ at $x$ and set:

$$F(x)=\{x_B \in S: (\exists n<\omega)(U^x_n \subset B) \}$$

Then $|F(x)| <\pi Nt(X)$. Indeed, suppose by contradiction that $|F(x)| \geq \pi Nt(X)$. Since $\pi Nt(X)$ has uncountable cofinality, by the pigeonhole principle there would be $n<\omega$ and $\mathcal{S} \subset \mathcal{B}$ such that $|\mathcal{S}|=\pi Nt(X)$ and $U^n_x \subset B$ for every $B \in \mathcal{S}$. But that is a contradiction. So by Lemma $\ref{hajnal}$ there is a set $D \subset S$ such that $|D|=|S|$ and $D$ is free for $F$. We claim that $D$ is discrete in $X$. Indeed, let $x_B \in D$ , and suppose by contradiction that there is a point $x \in D \cap B \setminus \{x_B\}$. By $D \cap F(x) \subset \{x\}$ we have $x_B \notin F(x)$. So $U^x_n \nsubseteq B$ for every $n<\omega$, which implies $x \notin B$ and that is a contradiction.
\end{proof}

\begin{corollary}
Let $L$ be a $\kappa$-Suslin Line, where $\kappa<\aleph_\omega$. Then $\pi Nt(L)=\kappa^+$.
\end{corollary}

\begin{proof}
By Theorem $\ref{upperbound}$ and Lemma $\ref{lowerbound}$ we have that $\aleph_1 \leq \pi Nt(L) \leq \kappa^+ < \aleph_\omega$. Thus $cf(\pi Nt(L))$ is uncountable. If $\pi Nt(L) <\kappa^+=d(L)$, then by Theorem $\ref{largediscrete}$, $L$ would contain a discrete set of size $\kappa^+$. But that is a contradiction because the supremum of the sizes of the discrete sets coincides with the cellularity in linearly ordered spaces.
\end{proof}

\begin{question}
Is it true that if $\kappa$ is any cardinal and $L$ is a $\kappa$-Suslin Line then $\pi Nt(L)=\kappa^+$?
\end{question}

\section{Chang's Conjecture for $\aleph_\omega$ and the Noetherian type of box products}

We denote by $(X^\kappa, \tau_{<\mu})$ the topology on $X^\kappa$ generated by the $<\mu$-supported boxes, that is the sets of the form $\prod_{\alpha < \kappa} W_\alpha$, where $W_\alpha$ is open in $X$ and $|\{\alpha < \kappa: W_\alpha \neq X \}|<\mu$.

L. Soukup (\cite{S}, see also \cite{KMS}) proved that if Chang's Conjecture for $\aleph_\omega$ holds then $Nt((2^{\aleph_\omega}, \tau_{<\omega_1})) \geq \aleph_2$. We refine this theorem by replacing the two-point discrete space $2$ with a much more general class of spaces. 

For cardinals $\lambda, \kappa, \mu, \tau$, the statement $(\lambda, \kappa) \twoheadrightarrow (\mu, \tau)$ means that for every structure $(A,B, \dots)$ with a countable signature such that $|A|=\lambda$ and $|B|=\kappa$, there is an elementary substructure $(A',B', \dots) \prec (A,B, \dots)$ such that $|A'|=\mu$ and $|B'|=\tau$. Chang's Conjecture for $\aleph_\omega$ is the assertion that $(\aleph_{\omega+1}, \aleph_\omega) \twoheadrightarrow (\aleph_1, \aleph_0)$. The consistency of Chang's Conjecture for $\aleph_\omega$  with the GCH was proved in \cite{LMS} assuming the consistency of a cardinal slightly larger than a huge cardinal. 

In this and the following section we will make use of the following combinatorial tool, which was introduced in \cite{KMS}. 

\begin{definition}
\cite{KMS} A family $\mathcal{F} \subset [\kappa]^\omega$ is called $(\kappa, \mu)$-sparse if for every set $\mathcal{G} \subset \mathcal{F}$ such that $|\mathcal{G}| \geq \kappa$ we have $|\bigcup \mathcal{G}| \geq \mu$.
\end{definition}

\begin{lemma} \label{LemmaKMS} \cite{KMS} {\ \\}
\begin{enumerate}
\item If $\square_{\aleph_\omega}$ and $\aleph_\omega^\omega=\aleph_{\omega+1}$ hold then $([\aleph_\omega]^\omega, \subseteq)$ contains an $(\aleph_1, \aleph_1)$-sparse cofinal family.
\item (L. Soukup, \cite{S}) Under Chang's Conjecture for $\aleph_\omega$ there are no $(\aleph_1, \aleph_1)$-sparse cofinal families in $([\aleph_\omega]^\omega, \subseteq)$.
\item Under Martin's Maximum $([\aleph_\omega]^\omega, \subseteq)$ contains a $(\aleph_2, \aleph_1)$-sparse cofinal family.
\item $([\aleph_\omega]^\omega, \subseteq)$ contains a $(\aleph_4, \aleph_1)$-sparse cofinal family.
\end{enumerate}
\end{lemma}

Given a family $\mathcal{U}$ of subsets of $X$, define $ord(x, \mathcal{U})=|\{U \in \mathcal{U}: x \in U \}|$.

\begin{lemma} \label{mainlemma}

Let $Y$ be a regular space and $X$ be a dense subspace of $(Y^{\aleph_\omega}, \tau_{<\omega_1})$. Let $\lambda$ be a cardinal such that $\lambda \leq \min \{cov(\aleph_\omega, \omega), |X|\}$. Let $\mathcal{U}$ be any base for $X$. Then there is a point $x \in X$ such that $ord(x, \mathcal{U}) \geq \lambda$.
\end{lemma}

\begin{proof}
Suppose by contradiction that $X$ has a base $\mathcal{U}$ such that $ord(x, \mathcal{U})<\lambda$ for every $x \in X$.

\noindent {\bf Claim 1:} If $x$ is not isolated, then there is a set $S \subset X \setminus \{x\}$ such that $|S| < \lambda$ and $x \in \overline{S}$.

\begin{proof}[Proof of Claim 1] Simply note that the character of every point in $X$ is strictly less than $\lambda$.
\renewcommand{\qedsymbol}{$\triangle$}
\end{proof}

Let $\mathcal{F}$ be an $(\aleph_4, \aleph_1)$-sparse cofinal family in $([\aleph_\omega]^\omega, \subseteq)$. Let $\mathcal{B}$ be the set of all countably supported boxes in $(Y^{\aleph_\omega}, \tau_{<\omega_1})$ and let $\mathcal{B}'$ be the following base: $$\mathcal{B}'= \{B \cap X:  supp (B) \in \mathcal{F} \textrm{ and } (\forall \alpha \in supp(B))(Int (\overline{\pi_{\alpha} (B)})=\pi_\alpha(B) \}.$$ $\mathcal{B}'$ is a base for the product because $X$ is regular. For $x \in X$, let $\mathcal{U}_x= \{U \in \mathcal{U}: x \in U \}$. Then $|\mathcal{U}_x|<\lambda$ for every $x \in X$. Define a local base $\mathcal{B}_x$ at $x$ in $X$ as follows:

$$\mathcal{B}_x =\{B \cap X: B \in \mathcal{B} \wedge (\exists U \in \mathcal{U})(x \in B \cap X \subset U)\}.$$ Given $\mathcal{A} \subset \mathcal{B}$ define $S(\mathcal{A})=\{supp(B): B \cap X \in \mathcal{A}\}$.

\vspace{.1in}

\noindent {\bf Claim 2:} $|S(\mathcal{B}_x)|<cov(\aleph_\omega, \omega)$, for every $x \in X$.

\begin{proof}[Proof of Claim 2] Suppose by contradiction that $|S(\mathcal{B}_x)| \geq cov(\aleph_\omega, \omega)$ for some $x \in X$. By regularity of $cov(\aleph_\omega, \omega)$ and $ord(x, \mathcal{U}) < \lambda \leq cov(\aleph_\omega, \omega)$ there would be a set $\mathcal{C} \subset \mathcal{B}_x$ such that $|S(\mathcal{C})| \geq cov(\aleph_\omega, \omega)$ and an open set $U \in \mathcal{U}$ such that $x \in B \cap X \subseteq U$ for every $B \in \mathcal{C}$. Now for every $B \in \mathcal{C}$ we can find $V_B \in \mathcal{U}$ such that $x \in V_B \subseteq B \cap X \subseteq U$. Since $ord(x, \mathcal{U})<\lambda$ there is $\mathcal{D} \subset \mathcal{C}$ and a fixed $V \in \mathcal{U}$ such that $|S(\mathcal{D})| \geq cov(\aleph_\omega, \omega)$ and $x \in V \subset B \cap X \subset U$. Let now $B'$ be a box such that $B' \cap X \in \mathcal{B}$ and $x \in B' \cap X \subset V \subset B \cap X \subset U$. By the definition of $\mathcal{B}$ and the density of $X$ we have that $B' \subset B$, and hence $supp(B) \subset supp(B')$ for every $B \in \mathcal{C}$, but this contradicts the fact that $\mathcal{F}$ is $(\aleph_4, \aleph_1)$-free.
\renewcommand{\qedsymbol}{$\triangle$}
\end{proof}

\noindent {\bf Claim 3:} Let $\{x_\alpha: \alpha < \kappa\}$ be an arbitrary enumeration of $X$ and fix $\beta < \lambda \leq \kappa$. If $\gamma \geq \beta$ then $x_\gamma \notin \overline{\{x_\alpha: \alpha < \beta \}}$.

\begin{proof}[Proof of Claim 3] From Claim 2 and the regularity of $cov(\aleph_\omega, \omega)$ it follows that 
$$\left |\bigcup_{\alpha < \beta} S(\mathcal{B}_{x_\alpha})\right |<cov(\aleph_\omega, \omega)$$

Hence there is $A \in [\aleph_\omega]^\omega$ such that $A$ is not covered by the support of any box $W$ such that $W \cap X \in \bigcup_{\alpha < \beta} \mathcal{B}_{x_\alpha}$. Let $G= B_1 \cap X \in \mathcal{B}$ be such that $x_\gamma \in G$ and $B_1$ is a box with $supp(B_1)=A$. Then there is $U \in \mathcal{U}$ such that $x_\gamma \in U \subset G$. For some $H=B_2 \cap X \in \mathcal{B}$ we have $x_\gamma \in H \subset U \subset G$. Now by the density of $X$ and the definition of $\mathcal{B}$ we have that $B_2 \subset B_1$ and hence $A=supp(B_1) \subset supp(B_2)$. Therefore $H \notin \bigcup_{\alpha < \beta} \mathcal{B}_{x_\alpha}$ and hence $H$ is an open neighbourhood of $x_\gamma$ in $X$ such that $H \cap \{x_\alpha: \alpha < \beta\} =\emptyset$.
\renewcommand{\qedsymbol}{$\triangle$}

\end{proof}

\noindent {\bf Claim 4:} Every subset of $X$ of size strictly less than $\lambda$ is closed discrete.

\begin{proof}[Proof of Claim 4] Let $F \subset X$ be such that $|F|<\lambda$ and let $\{x_\alpha: \alpha < \kappa \}$ be an enumeration of $X$ such that $F$ coincides with one of its initial segments.  Then by Claim 2, $F$ is closed. This proves that every $<\lambda$-sized subset of $X$ is closed. But then each subset of $F$ is closed and therefore $F$ is also discrete
\renewcommand{\qedsymbol}{$\triangle$}

\end{proof}

Comparing Claims 1 and 4 one gets a contradiction. Thus, there is $x \in X$ such that $ord(x, \mathcal{U}) \geq \lambda$.
\end{proof}

A point $x \in X$ is called a $P_{\omega_1}$-point if for every family $\mathcal{U}$ of $\leq \aleph_1$ many neighbourhoods of $x$, the set $\bigcap \mathcal{U}$ has non-empty interior. Every isolated point is of course a $P_{\omega_1}$-point.

\begin{theorem} \label{mainthm}
Assume Chang's Conjecture for $\aleph_\omega$. Let $Y$ be a regular space having a dense set of $P_{\omega_1}$-points. Let $X$ be the space $(Y^{\aleph_\omega}, \tau_{<\omega_1})$. Then $Nt(X) \geq \aleph_2$.
\end{theorem}

\begin{proof}
Let $\mathcal{B}$ be a base for $X$ and $Z \subset Y$ be the set of all $P_{\omega_1}$-points. Use Chang's Conjecture for $\aleph_\omega$ to find, for large enough regular $\theta$, an elementary submodel $M \prec H(\theta)$ such that:

$$\{\aleph_\omega, \aleph_{\omega+1}, Y, Z, \mathcal{B}\} \subset M, |M \cap \aleph_{\omega+1}|=\aleph_1 \textrm{ and } |M \cap \aleph_\omega|=\aleph_0.$$

 Since $Z^{\aleph_\omega} \in M$ is dense in $X \in M$ with the induced topology and $\aleph_{\omega+1}\leq cov(\aleph_\omega, \omega)= \min \{|Z^{\aleph_\omega}|, cov(\aleph_\omega, \omega)\}$, by Lemma $\ref{mainlemma}$, we can fix a point $f \in Z^{\aleph_\omega} \cap M$ such that $\{B \in \mathcal{B}: f \in B \}$ has size at least $\aleph_{\omega+1}$. So let $F: \aleph_{\omega+1} \to I:=\{B \in \mathcal{B}: f \in B \}$ be a one-to-one function. By elementarity, we can assume that $F \in M$. Now $F \upharpoonright M$ is a one-to-one function between an uncountable set and $I \cap M$, which shows that $I \cap M$ is uncountable. Let $J \subset I \cap M$ be a set of size $\aleph_1$. For every $B \in J$ fix a box $W \in M$ such that $f \in W \subset B$. Let $\mathcal{W}$ be the set of all such boxes. Since $supp(W) \in M$ and $|supp(W)| \leq \aleph_0$, we have $supp(W) \subset M$ for every $W \in \mathcal{W}$. For every $\alpha \in M \cap \aleph_\omega$ let now $B_\alpha \subset \bigcap_{W \in \mathcal{W}} \pi_\alpha(W)$ be a non-empty open set. We can fix such a set because $\pi_\alpha(f) \in \pi_\alpha(W)$ for every $W \in \mathcal{W}$ and $\pi_\alpha(f)$ is a $P_{\omega_1}$-point. Let now $B$ be the countably supported box defined as $\pi_\alpha(B)=B_\alpha$ if and only if $\alpha \in M \cap \aleph_\omega$ and $\pi_\alpha(B)=Y$ otherwise. Then $B \subset W$ for every $W \in \mathcal{W}$, which proves that $\mathcal{B}$ is not $\omega_1^{op}$-like.
\end{proof}

Recall that a space is scattered if each of its subsets has an isolated point. Every scattered space has a dense set of isolated points, so Theorem $\ref{mainthm}$ has the following corollary.

\begin{corollary} \label{corscattered}
Assume Chang's Conjecture for $\aleph_\omega$. Let $X$ be a regular scattered space. Then $Nt((X^{\aleph_\omega}, \tau_{<\omega_1})) \geq \aleph_2$.
\end{corollary}

\begin{corollary}
(L. Soukup, \cite{S}) $Nt((2^{\aleph_\omega}, \tau_{<\omega_1})) \geq \aleph_2$ under Chang's Conjecture for $\aleph_\omega$.
\end{corollary}

In \cite{KMS} the authors provided a model of Chang's Conjecture for $\aleph_\omega$ where the Noetherian $\pi$-type of $(2^{\aleph_\omega}, \tau_{<\omega_1})$ is $\aleph_1$, so it is not clear whether the Noetherian $\pi$-type of $(2^{\aleph_\omega}, \tau_{<\omega_1})$ can be pushed above $\aleph_1$. However, the Noetherian $\pi$-type of the small $\sigma$-product of such a space is indeed influenced by Chang's conjecture.

Let $\sigma(2^{\aleph_\omega})=\{x \in 2^{\aleph_\omega}: |x^{-1}(1)|<\omega\}$.

\begin{theorem} \label{ChangPi}
Suppose that Chang's Conjecture for $\aleph_\omega$ holds. Let $X$ be the space $\sigma(2^{\aleph_\omega})$ with the topology induced by $(2^{\aleph_\omega}, \tau_{<\omega_1})$. Then $\pi Nt(X) \geq \aleph_2$.
\end{theorem}

\begin{proof}
Note that $\pi w(X)=cov(\aleph_\omega, \omega) \geq \aleph_{\omega+1}$ while $|X|=\aleph_\omega$. Therefore, for every $\pi$-base $\mathcal{P}$ there is a point $x \in X$ such that $|\{P \in \mathcal{P}: x \in P\}| \geq \aleph_{\omega+1}$. Now, an argument similar to the proof of Theorem $\ref{mainthm}$ shows that $\pi Nt(X) \geq \aleph_2$.
\end{proof}

\begin{corollary} \label{closedsubset}
There are models of ZFC where the Noetherian $\pi$-type of $(2^{\aleph_\omega}, \tau_{<\omega_1})$ is $\aleph_1$, but $(2^{\aleph_\omega}, \tau_{<\omega_1})$ contains a closed subset of Noetherian $\pi$-type $>\aleph_1$.
\end{corollary}

\begin{proof}
In Theorem 3.8 of \cite{KMS} the authors proved that adding $\aleph_{\omega+1}$ Cohen reals to any model of Chang's Conjecture for $\aleph_\omega$ one gets a model where $\pi Nt((2^{\aleph_\omega}, \tau_{<\omega_1}))=\aleph_1$ and Chang's Conjecture for $\aleph_\omega$ still holds. By Theorem $\ref{ChangPi}$ $\pi Nt(\sigma(2^{\aleph_\omega}))=\aleph_2$ in that model and $\sigma(2^{\aleph_\omega})$ embeds as a closed subset into $(2^{\aleph_\omega}, \tau_{<\omega_1})$.
\end{proof}

Note that in the model of Corollary $\ref{closedsubset}$ the Noetherian type of $(2^{\aleph_\omega}, \tau_{<\omega_1})$ is $\aleph_2$.
However, to get a model of ZFC where the Noetherian type of $(2^{\aleph_\omega}, \tau_{<\omega_1})$ is $\aleph_1$ but $(2^{\aleph_\omega}, \tau_{<\omega_1})$ contains a closed subset of Noetherian type $>\aleph_1$ it suffices to take any model of $\square_{\aleph_\omega}$ and $\aleph_\omega^\omega=\aleph_{\omega+1}$.
Indeed, $(2^{\aleph_\omega}, \tau_{<\omega_1})$ has Noetherian type $\aleph_1$ in that model by Lemma $\ref{LemmaKMS}$ and Corollary $\ref{corequiv}$ and it contains a closed copy of the one-point Lindel\"ofication of a discrete set of size $\aleph_2$, which has Noetherian type $\aleph_3$ according to Theorem $\ref{lindelofication}$ below.

\section{On the Noetherian type of Pixley-Roy Hyperspaces}

It can be proved that $Nt(Exp(X)) \leq Nt(X)$, where $Exp(X)$ is the set of all finite subsets of $X$ with the Vietoris topology. This was essentially proved by Peregudov and Shapirovskii in \cite{PS}. We don't know, at the moment, of any space such that $Nt(Exp(X)) <Nt(X)$.
 
Let $(X, \tau)$ be a space and $PR(X)=[X]^{<\omega}$. We define a topology on $PR(X)$ by declaring the set $\{[F,U]: F \in PR(X), U \in \tau \}$ to be a base, where $[F,U]=\{G \in PR(X): F \subset G \subset U \}$. This is a well-known hyperspace construction called the \emph{Pixley-Roy topology on $X$} (see \cite{vD}). Hajnal and Juh\'asz remarked in \cite{HJ} that all cardinal functions on $PR(X)$ are easy to compute in terms of their values on $X$, except for the cellularity. We think that Noetherian type might be one more exception. 

We tested the following conjecture on several examples of spaces.

\begin{conjecture}
$Nt(PR(X)) \leq Nt(X)$.
\end{conjecture}

Our aim in this section is to show that the gap between $Nt(PR(X))$ and $Nt(X)$ can be as large as the gap between $\aleph_4$ and $\aleph_\alpha$ for countable $\alpha$.

Given cardinals $\kappa$ and $\mu$, we denote by $A_\mu(\kappa)$ the topology on $\kappa \cup \{p\}$ such that every point of $\kappa$ is isolated and a neighbourhood of $p$ is a set of the form $\{p\} \cup \kappa \setminus C$, where $|C|<\mu$.

\begin{theorem} \label{lindelofication}
Let $\kappa$ and $\mu$ be cardinals such that $cf(cov(\kappa, <\mu))>\mu$. Then $Nt(A_\mu(\kappa))=cov(\kappa,<\mu)^+$.

\end{theorem}

\begin{proof}
Note that $Nt(A_\mu(\kappa)) \leq (cov(\kappa, <\mu))^+$ since $w(A_\mu(\kappa))=cov(\kappa, <\mu)$ and the inequality $Nt(X) \leq w(X)^+$ is true for every space $X$.

For the reverse inequality let $\lambda =cov(\kappa, <\mu)$ and suppose by contradiction that $X$ has a $\lambda^{op}$-like base $\mathcal{B}$. For every $B \in \mathcal{B}$ such that $p \in B$ let $C_B \in [\kappa]^{<\mu}$ be such that $\kappa \setminus C_B \subset B$. Let $\mathcal{C}$ be the collection of all these $C_B$'s. By refining $\mathcal{C}$ we can assume it has size $cov(\kappa, <\mu)$. If we could find $\lambda$ many elements of $\mathcal{C}$ which miss some ordinal $\gamma$ then the isolated point $\gamma$ would have $\lambda$ many distinct supersets in $\mathcal{B}$, thus contradicting the fact that $\mathcal{B}$ is $\lambda^{op}$-like. Now, since $cf(cov(\kappa, <\mu))>\mu$, then for every $\alpha < \mu$ we can find $\beta_\alpha < \lambda$ such that $\alpha \in C_\gamma$ for every $\gamma \leq \beta_\alpha$. Let $\beta=\sup_{\alpha<\mu} \beta_\alpha<\lambda$. Then $\mu \subset C_{\beta+1}$ thus contradicting the fact that $C_{\beta+1}$ has size smaller than $\mu$.
\end{proof}

\begin{corollary}
Let $X$ be the one-point lindel\"ofication of a discrete space of size $\aleph_\omega$. Then $Nt(X)=cov(\aleph_\omega, \omega)^+$.
\end{corollary}

\begin{theorem} \label{pixleytheorem}
Let $X$ be the one-point lindel\"ofication of a discrete space of size $\mu$. Then $Nt(PR(X)) \leq \kappa$ if and only if there is a $(\kappa, \aleph_1)$-sparse cofinal family in $([\mu]^\omega, \subseteq)$.
\end{theorem}

\begin{proof}
Let $p$ be the unique non-isolated point of $X$ and note that if $F \in [X]^{<\omega}$ and $p \notin F$ then $F$ is an isolated point in the Pixley-Roy topology. Suppose that a $(\kappa, \aleph_1)$-sparse family $\mathcal{C}$ exists in $([\mu]^\omega, \subseteq)$. For every $F$ such that $p \in F$, let $\mathcal{C}_F=\{C \setminus F: C \in \mathcal{C} \}$. Then $\mathcal{C}_F$ is a $(\kappa, \aleph_1)$-sparse family covering every countable subset of $X$ missing $F$ and so $\mathcal{B}_F=\{[F, \mu \setminus C]: C \in \mathcal{C}_F \}$ is a local base at $F$. Let $\mathcal{B}=\bigcup \{\mathcal{B}_F: p \in F \} \cup \{\{F\}: p \notin F \}$. We claim that $\mathcal{B}$ is a $\kappa^{op}$-like base for $PR(X)$. Indeed, suppose by contradiction that $[F, \mu \setminus C] \subset [F_\alpha, \mu \setminus C_\alpha]$ for every $\alpha < \kappa$ and some family $\{C_\alpha: \alpha < \kappa \} \subset \mathcal{C}$. We have that $F_\alpha \subset F$ for every $\alpha < \kappa$, so we may assume that for some $G$ we have $F_\alpha=G$ and $C_\alpha\in \mathcal{C}_G$ for every $\alpha < \kappa$.  We have that $C_\alpha \subset C$ for every $\alpha$ and $G \subset F$. So $C_\alpha \cup G \subset C \cup F$ for every $\alpha < \kappa$, contradicting the fact that $\mathcal{C}$ is $(\kappa, \aleph_1)$-sparse.

Viceversa, suppose that $\mathcal{B}$ is a $\kappa^{op}$-like base for $X$ and let $F \in PR(X)$ be a point such that $p \in F$. Let $\{[F, \mu \setminus C]: C \in \mathcal{C} \}$ be a \emph{standard} local base at $F$ refining elements of $\mathcal{B}$. We claim that $\{C \cup F: C \in \mathcal{C} \}$ is a $(\kappa, \aleph_1)$-sparse cofinal family of countable subsets of $\mu$. Indeed, suppose by contradiction that $C_\alpha \cup F \subset C \cup F$, for every $\alpha < \kappa$. Then $F \cap (C_\alpha \cup C)=\emptyset$ implies that $C_\alpha \subset C$ and hence $[F, \mu \setminus C]\subset [F, \mu \setminus C_\alpha]$, which contradicts the fact that $\mathcal{B}$ is $\kappa^{op}$-like.
\end{proof}

\begin{corollary} \label{corequiv}
The following are equivalent:

\begin{enumerate}
\item \label{sparse} There is a cofinal $(\kappa, \aleph_1)$-sparse family on $([\mu]^\omega, \subseteq)$.
\item \label{box} $Nt((2^\mu, \tau_{<\omega_1})) \leq \kappa$.
\item \label{pixley} $Nt(PR(A_{\omega_1}(\mu)) \leq \kappa$.
\end{enumerate}
\end{corollary}

\begin{proof}
The equivalence $(\ref{sparse}) \iff (\ref{box})$ was proved in \cite{KMS} while the equivalence $(\ref{sparse}) \iff (\ref{pixley})$ follows from Theorem $\ref{pixleytheorem}$.
\end{proof}

From Lemma $\ref{LemmaKMS}$ and Theorem $\ref{pixleytheorem}$ the next corollary follows.

\begin{corollary} {\ \\}
\begin{enumerate}
\item \label{alephn} For every $\mu < \aleph_\omega$, $Nt(PR(A_{\omega_1}(\mu)) \leq \aleph_1$.
\item \label{aleph4} $Nt(PR(A_{\omega_1}(\aleph_\omega)) \leq \aleph_4$.
\item \label{chang} If Chang's Conjecture for $\aleph_\omega$ holds then $Nt(PR(A_{\omega_1}(\aleph_\omega)) \geq \aleph_2$.
\item \label{martin} If Martin's Maximum holds then $Nt(PR(A_{\omega_1}(\aleph_\omega)) \leq \aleph_2$.
\end{enumerate}
\end{corollary}

\begin{corollary}
The gap between $Nt(PR(X))$ and $Nt(X)$ can be as large as the gap between $\aleph_4$ and $\aleph_\alpha$ for any $\alpha<\omega_1$.
\end{corollary}

\begin{proof}
Assuming the consistency of a supercompact cardinal Shelah \cite{Sh} proved the consistency of \emph{$\aleph_\omega$ is strong limit and $2^{\aleph_\omega}=\aleph_{\alpha+1}$}, for every $\alpha<\omega_1$. It is easy to see that in the resulting model $2^{\aleph_\omega}=(\aleph_\omega)^\omega$ and since $(\aleph_\omega)^\omega=cov(\aleph_\omega, \omega) \cdot 2^\omega$ and $2^\omega<\aleph_\omega$ we have that $cov(\aleph_\omega, \omega)=\aleph_{\alpha+1}$.  So it suffices to take $X$ to be the one-point lindel\"ofication of a discrete set of size $\aleph_\omega$.
\end{proof}

\begin{question}
Is there a space $X$ such that the gap between $Nt(X)$ and $Nt(PR(X))$ is not bounded in ZFC?
\end{question}

\end{document}